\documentclass[11pt]{article}

\usepackage{amsmath,amssymb,amsthm,setspace,graphicx,array,mathtools}
\usepackage[text={6.5in,8.4in}, top=1.3in]{geometry}

\newtheorem{theorem}{Theorem}[section]
\newtheorem{proposition}[theorem]{Proposition}
\newtheorem{lemma}[theorem]{Lemma}
\newtheorem{corollary}[theorem]{Corollary}
\newtheorem{conjecture}[theorem]{Conjecture}

\newtheorem{definition}[theorem]{Definition}

\newcommand{\D}{{\rm disc}}

\begin{document}
	
\title{
	Packing and covering a given directed graph in a directed graph
}
	
\author{
	Raphael Yuster
	\thanks{Department of Mathematics, University of Haifa, Haifa 3498838, Israel. Email: raphael.yuster@gmail.com\;.}
}
	
\date{}
	
\maketitle
	
\setcounter{page}{1}
	
\begin{abstract}
	For every fixed $k \ge 4$, it is proved that if an $n$-vertex directed graph has at most $t$ pairwise arc-disjoint directed $k$-cycles, then there exists a set of at most $\frac{2}{3}kt+ o(n^2)$ arcs that meets all directed $k$-cycles and that the set of $k$-cycles admits a fractional cover of value at most $\frac{2}{3}kt$. It is also proved that the ratio $\frac{2}{3}k$ cannot be improved to a constant smaller than $\frac{k}{2}$. For $k=5$ the constant $2k/3$ is improved to $25/8$ and for $k=3$ it was recently shown by Cooper et al. that the constant can be taken to be $9/5$. The result implies a deterministic polynomial time $\frac{2}{3}k$-approximation algorithm for the directed $k$-cycle cover problem, improving upon a previous $(k{-}1)$-approximation algorithm of Kortsarz et al.
	
	More generally, for every directed graph $H$ we introduce a graph parameter $f(H)$ for which it is proved that if an $n$-vertex directed graph has at most $t$ pairwise arc-disjoint $H$-copies, then there exists a set of at most $f(H)t+ o(n^2)$ arcs that meets all $H$-copies and that the set of $H$-copies admits a fractional cover of value at most $f(H)t$. It is shown that for almost all $H$ it holds that $f(H) \approx |E(H)|/2$ and that for every $k$-vertex tournament $H$ it holds that $f(H) \le \lfloor k^2/4 \rfloor$.
	
\vspace*{3mm}

\end{abstract}

\section{Introduction}

Let $H$ be a directed or undirected graph. For a directed (or undirected) multigraph $G$, let $\nu_H(G)$ denote the maximum number of pairwise arc-disjoint (edge-disjoint) copies of $H$ in $G$ and let  $\tau_H(G)$ denote the minimum number of arcs (edges) whose removal from $G$ results in a subgraph with no copies of $H$.
The fractional versions of these parameters (see Section \ref{sec:prelim} for a definition) are denoted by $\nu^*_H(G)$ and $\tau^*_H(G)$, respectively. It is readily observed that
$\tau_H(G) \ge \tau^*_H(G)=\nu^*_H(G) \ge \nu_H(G)$ and that $\tau_H(G) \le |E(H)|\nu_H(G)$.
These parameters can also be naturally extended to the weighted setting where each arc (edge) of $G$ is assigned a non-negative weight (see Section \ref{sec:prelim} for a definition).

The undirected case has substantial literature. The starting point of these problems is the well-known and yet unsolved conjecture of Tuza \cite{tuza-1990} asserting that $\tau_{C_3}(G)\le 2\nu_{C_3}(G)$ for every undirected graph $G$. Stated equivalently, the conjecture asserts that if a graph has at most $t$
pairwise edge-disjoint triangles, then it can be made triangle-free by removing at
most $2t$ edges. The best known upper bound is by Haxell \cite{haxell-1999} who proved
that $\tau_{C_3}(G)\le 2.87\nu_{C_3}(G)$. Krivelevich \cite{krivelevich-1995} proved  a
fractional version of Tuza's conjecture, namely that $\tau_{C_3}(G)\le 2\nu^*_{C_3}(G)$ (he also proved that $\tau^*_{C_3}(G)\le 2\nu_{C_3}(G)$).
It was later observed in \cite{yuster-2012} that using a method of Haxell and R\"odl \cite{HR-2001},
Krivelevich's result implies that Tuza's conjecture asymptotically holds in
the dense setting, specifically $\tau_{C_3}(G)\le 2\nu_{C_3}(G)+o(n^2)$ where $n$ is the number of the vertices of $G$. There are examples showing that the constant $2$ in Tuza's conjecture cannot be
replaced with a smaller one, even in the dense setting \cite{BK-2016}.

The aforementioned results concerning $C_3$ have some nontrivial generalizations to additional graphs.
In \cite{yuster-2012} the author proved that $\tau_{K_k}(G)\le \lfloor k^2/4 \rfloor \nu^*_{K_k}(G)$ 
and that $\tau_{K_k}(G)\le \lfloor k^2/4 \rfloor \nu_{K_k}(G)+o(n^2)$. This is presently the best
known upper bound for the case of $K_k$. Kortsarz, Langberg, and Nutov \cite{KLN-2010} proved
that $\tau_{C_k}(G)\le (k-1) \nu^*_{C_k}(G)$. Their main motivation came from the related well-known natural optimization problem.
\begin{definition}[The $H$-cover problem]
	Let $H$ be a fixed (directed) graph. Given a (directed) graph $G$, find a minimum size subset of edges (arcs) of $G$ whose removal results in an $H$-free subgraph of $G$.
\end{definition}
It is well-known \cite{yannakakis-1978} that $H$-cover is NP-hard already for some small $H$ (e.g. $H=K_3$) thus we seek a polynomial time approximation algorithm. One may similarly define the $H$-cover problem
in the weighted setting where the goal is to find a subset of edges (arcs) that covers all $H$-copies and
whose total weight is the minimum possible.
The proof in \cite{KLN-2010}, as well as Krivelevich's proof for $C_3$, give a polynomial
time $(k{-}1)$-approximation algorithm for $C_k$-cover.
Similarly, the proof in \cite{yuster-2012} can be shown to give a polynomial time $\lfloor k^2/4 \rfloor$-approximation algorithm for $K_k$-cover.

\vspace{0.3cm}
In this paper we consider the directed case, which has recently gained attention.
Already when posing his conjecture, Tuza \cite{tuza-1990} asked whether
$\tau_{\overrightarrow{C_3}}(D)\le 2\nu_{\overrightarrow{C_3}}(D)$ where $\overrightarrow{C_k}$
denotes the directed cycle on $k$ vertices and $D$ is a directed graph.
McDonald, Puleo and Tennenhouse \cite{MPT-2020} answered
Tuza's question affirmatively proving that $\tau_{\overrightarrow{C_3}}(D)\le 2\nu_{\overrightarrow{C_3}}(D)-1$ for any directed multigraph $D$. In fact, they conjectured that a significantly stronger variant of Tuza's conjecture holds in the $\overrightarrow{C_3}$ case.
Specifically, they conjectured that $\tau_{\overrightarrow{C_3}}(D)\le 1.5\nu_{\overrightarrow{C_3}}(D)$
for any directed multigraph $D$. They also gave an example showing that if true, the constant $1.5$ is best
possible.
Recently, Cooper et al. \cite{CGKK-2022} proved that the fractional version for
$\overrightarrow{C_3}$ satisfies a factor better than $2$. Specifically,
$\tau_{\overrightarrow{C_3}}(D)\le 1.8\nu_{\overrightarrow{C_3}}^*(D)$ for any
arc-weighted directed multigraph $D$. As in the undirected cases mentioned above, this also implies that
$\tau_{\overrightarrow{C_3}}(D)\le 1.8\nu_{\overrightarrow{C_3}}(D)+o(n^2)$ for any
unweighted directed graph $D$. In their paper \cite{KLN-2010} mentioned above, Kortsarz, Langberg, and Nutov
stated and showed that $\tau_{\overrightarrow{C_k}}(D)\le (k-1)\nu_{\overrightarrow{C_k}}(D)$ for
all $k \ge 3$ and that the $\overrightarrow{C_k}$-cover problem admits a polynomial time
$(k{-}1)$-approximation algorithm.

Our main result gives a general upper bound for $\tau_H(D)$ in terms of $\nu^*_H(D)$ that applies
to {\em any} fixed directed graph $H$ and to any directed weighted multigraph $D$.
However, as a special case of our result implies an improvement of the aforementioned result
for $\overrightarrow{C_k}$ for all $k \ge 4$, we prefer to first state our results for directed $k$-cycles.
To simplify some notation we use the subscript $k$ instead of the subscript $\overrightarrow{C_k}$
in the parameter definitions.

\begin{theorem}\label{t:main}
	If $D$ is an arc-weighted directed multigraph, then $\tau_k(D) \le (2k/3)\nu_k^*(D)$.
	For $k=5$ we further have $\tau_5(D) \le (25/8)\nu_5^*(D)$.
\end{theorem}
\noindent Note that for $k=3$ the result in \cite{CGKK-2022} gives a better constant, but already for $k \ge 4$
this improves upon the state of the art. Our proof implies a deterministic approximation algorithm.
\begin{corollary}\label{coro:approx}
	The $\overrightarrow{C_k}$-cover problem (also in the weighted multigraph setting) admits a deterministic polynomial time $(2k/3)$-approximation algorithm. For $k=5$ the approximation ratio is $25/8$.
\end{corollary}
\noindent As in \cite{CGKK-2022}, this will also imply a non-fractional result in the dense setting.
\begin{corollary}\label{coro:dense}
	If $D$ is an $n$-vertex directed graph, then $\tau_k(D) \le (2k/3)\nu_k(D)+o(n^2)$
	and $\tau_5(D) \le (25/8)\nu_5(D)+o(n^2)$.
\end{corollary}

Given Theorem \ref{t:main} and its corollaries, it is of interest to ask whether the constant $2k/3$
(and $25/8$ when $k=5$) can be improved. We conjecture that it can.
\begin{conjecture}\label{conj:1}
	Let $k \ge 3$ be fixed. For all $n$ sufficiently large, if $D$ is an $n$-vertex directed graph,
	then $\tau_k(D) \le (k/2)\nu_k(D)$.
\end{conjecture}
Note that the case $k=3$ of Conjecture \ref{conj:1} is the aforementioned conjecture of
McDonald, Puleo and Tennenhouse \cite{MPT-2020}. The constant $k/2$ in Conjecture \ref{conj:1} cannot be made smaller. In  fact, it cannot be made smaller even if the host graph is a regular tournament.
\begin{theorem}\label{t:lower}
	Let $k \ge 3$ be fixed. For all $n$ sufficiently large satisfying $n \equiv 1 \pmod {2k}$
	there is a regular $n$-vertex tournament $T$ such that $\nu_k(T)=\nu^*_k(T)=n(n-1)/2k$
	and $\tau_k(T) = n^2/4-o(n^2)$.
\end{theorem}

Generalizing Theorem \ref{t:main} to arbitrary $H$ requires introducing a graph parameter.
For a directed graph $L$, {\em the blowup} of $L$, denoted by $B(L)$, is obtained by replacing each vertex $v \in V(L)$ with a countably infinite independent set $I_v$, and having all possible arcs from $I_a$ to $I_b$ whenever $(a,b) \in E(L)$. Let $\D_H(L)$ denote the minimum number of arcs that should be
added to $B(L)$ so that a copy of $H$ is obtained. Let
\begin{align}
	f(H,L) & = \max \left\{ |E(H)|\left(1-\frac{|E(L)|}{|V(L)|^2}\right)\,,\,|E(H)|-\D_H(L)  \right\}	\label{e:fhl}\\
	f(H) & = \inf_{L} f(H,L)	\label{e:fh}
\end{align}
where the infimum is taken over all nonempty directed graphs $L$.
Notice that $f(H)$ is a certain measure of how much $H$ embeds in a blowup of any possible directed graph. Our main result follows.
\begin{theorem}\label{t:main-2}
	If $D$ is an arc-weighted directed multigraph, then $\tau_H(D) \le f(H)\nu_H^*(D)$.
\end{theorem}
It is possible to provide good upper bounds, and sometimes determine $f(H)$ for some particular $H$ or certain families of directed graphs.
In fact, in many cases (but {\em not} all cases) the infimum in \eqref{e:fh} is a minimum, so that
$f(H)=f(H,L)$ is attained by some $L$. As we show in Section \ref{sec:fH}, 
$f(\overrightarrow{C_k})=2k/3$ except when $k=2$ in which case $f(\overrightarrow{C_2})=1$
or $k=5$ in which case $f(\overrightarrow{C_5})=25/8$. Thus, Theorem \ref{t:main} is a corollary of
Theorem \ref{t:main-2}. As another example, $f(H)\le \lfloor k^2/4\rfloor$ for all
$k$-vertex tournaments. As we show, this implies the known undirected results
$\tau_{K_k}(G) \le \lfloor k^2/4\rfloor\nu^*_{K_k}(G)$ \cite{krivelevich-1995,yuster-2012} also for the weighted multigraph setting. In all of these cases, the values are attained by some $L$.
The following proposition shows that almost all oriented graphs have $f(H)$ no larger than about
half of the size of their arc set.
\begin{proposition}\label{prop:1}
	Let $G$ be an undirected graph with $n$ vertices and $\Omega(n\ln n)$ edges.
	Let $H$ be a randomly chosen orientation of $G$.
	Then, asymptotically almost surely, $f(H) = (1+o_n(1))|E(H)|/2$. In particular, $\tau_H(D) \le (1+o_n(1))|E(H)|\nu_H^*(D)/2$ asymptotically almost surely.
\end{proposition}
Finally, Corollaries \ref{coro:approx} and \ref{coro:dense} are, in fact, special cases of the following more general corollaries of Theorem \ref{t:main-2}.
\begin{corollary}\label{coro:approx-2}
	The problem of determining $\tau_H(D)$ admits a deterministic polynomial time $f(H)$-approximation algorithm. For any nonempty directed graph $L$, the $H$-cover problem
	(also in the weighted multigraph setting) admits a deterministic polynomial time $f(H,L)$-approximation algorithm. In particular, if $f(H)=f(H,L)$ for some $L$, then the $H$-cover problem
	admits a deterministic polynomial time $f(H)$-approximation algorithm.
\end{corollary}

\begin{corollary}\label{coro:dense-2}
	If $D$ is an $n$-vertex directed graph, then $\tau_H(D) \le f(H)\nu_H(D)+o(n^2)$.
\end{corollary}

The rest of this paper is organized as follows. Some required definitions and
lemma are given in Section \ref{sec:prelim}. In Section \ref{sec:fH} we determine $f(H)$
for $k$-cycles and some other special directed graphs and prove Proposition \ref{prop:1}.
The proof of Theorem \ref{t:main-2} is given in Section \ref{sec:proof}.
Theorem \ref{t:lower} is proved in Section \ref{sec:lower}.

\section{Preliminaries}\label{sec:prelim}

We set notation used throughout the paper. For a directed
(multi)graph $D$, let $V(D)$ denote its vertex set and $E(D)$ denote its arc set.
Directed graphs are allowed to contain directed cycles of length $2$ and directed multigraphs
are also allowed to contain more than one arc in the same direction between two vertices.
An {\em orientation} of an undirected graph is a directed graph obtained by orienting each edge
in one of the possible directions. Equivalently, it is a directed graph with no directed cycles of length $2$.
A {\em tournament} is an orientation of the complete graph. A directed graph is {\em acyclic} if it has no directed cycles and it is $H$-free if it has no subgraph that is isomorphic to $H$.
Let $T_k$ denote the unique transitive (i.e. acyclic) tournament on $k$ vertices.

For a directed graph $H$ we denote by $C(H,D)$ the set of all subgraphs of $D$ isomorphic to $H$
(namely, the set of $H$-copies in $D$). If $F \subseteq E(D)$, then $D \setminus F$ is the spanning subgraph
of $D$ obtained by removing the arcs in $F$. If $F=\{e\}$ we use the shorthand $D \setminus e$.
We say that $D$ is {\em arc-weighted} if every arc $e$ is assigned a non-negative weight $w(e)$.

A {\em fractional $H$-packing} of an arc-weighted directed multigraph $D$ is a
function $m : C(H,D) \rightarrow [0,\infty)$ such that for every arc $e \in E(D)$, the sum of
$m(X)$ taken over all $H$-copies in $D$ that contain $e$ is at most $w(e)$. The {\em value} of $m$ is
the sum of $m(X)$ taken over all $H$-copies. The maximum
value of a fractional $H$-packing of $D$ is denoted by $\nu_H^*(D)$.
If $D$ is unweighted (equivalently, all arc weights are $1$) and $m(X) \in \{0,1\}$ for each
$X \in C(H,D)$ we say that $m$ is an {\em $H$-packing}.
The maximum value of an $H$-packing of an unweighted directed multigraph $D$ is denoted by $\nu_H(D)$. Equivalently, $\nu_H(D)$ is the maximum number of pairwise arc-disjoint $H$-copies in $D$.
Clearly, $\nu_H(D) \le \nu_H^*(D)$ for every unweighted directed multigraph $D$.

A {\em fractional $H$-cover} of an arc-weighted directed multigraph $D$ is a
function $c : E(D)\rightarrow [0,1]$ such that for each $X \in C(H,D)$, the sum of the values
of $c$ on the arcs of $X$ is at least $1$. The {\em value} of $c$ is the sum of
$w(e)c(e)$ taken over all arcs $e \in E(D)$. The minimum value of the fractional
$H$-cover of $D$ is denoted by $\tau_H^*(D)$.
If $c(e) \in \{0,1\}$ for each $e \in E(D)$ we say that $c$ is an {\em $H$-cover}.
The minimum value of an $H$-cover is denoted by $\tau_H(D)$. Equivalently, $\tau_H(D)$ is the minimum sum of weights of a set of arcs $F$ such that $D \setminus F$  is $H$-free.
Clearly, $\tau_H(D) \ge \tau_H^*(D)$ for every arc-weighted directed multigraph $D$.

Given an arc-weighted directed multigraph $D$, a minimum value fractional $H$-cover of $D$ and a maximal value fractional $H$-packing of $D$ can be computed in polynomial time by linear programming.
Moreover, by linear programming duality, $\nu_H^*(D)=\tau_H^*(D)$. In particular, $\tau_H(D) \ge \nu_H^*(D)$ and if $D$ is unweighted then $\tau_H(D) \ge \tau_H^*(D)=\nu_H^*(D) \ge \nu_H(D)$.

Suppose now that $D$ is an unweighted directed graph.
It is not difficult to provide examples where $\tau_H(D)$ is larger than $\tau^*_H(D)$ and
to provide examples where $\nu_H(D)$ is smaller than $\nu^*_H(D)$.
However, in a dense setting, the latter pair are always close.
The following result of Nutov and Yuster \cite{NY-2007} is a directed version of a result of the author
\cite{yuster-2005} which, in turn is a generalization of a result of Haxell and R\"odl \cite{HR-2001} on the difference between a fractional and integral packing in undirected graphs.
\begin{lemma}\label{l:ny}
	Let $H$ be a fixed directed graph. If $D$ is a directed graph with $n$ vertices, then
	$\nu^*_H(D) \le \nu_H(D)+o(n^2)$. Furthermore, there exists a polynomial time algorithm
	that produces an $H$-packing of $D$ of size at least $\nu^*_H(D)-o(n^2)$. \qed
\end{lemma}
\noindent Corollary \ref{coro:dense-2} follows immediately from Lemma \ref{l:ny} and Theorem \ref{t:main-2}.

\section{$f(H)$ and $f(H,L)$}\label{sec:fH}

In this section we consider $f(H)$ and $f(H,L)$; we determine $f(H)$ for certain families of directed graphs
and certain small $H$ and provide some general upper bounds for it.
To avoid trivial cases, we assume that $H$ is a directed graph with at least two arcs and that $L$
is a nonempty directed graph with $r \coloneqq |V(L)|$ vertices.
\begin{proposition}\label{prop:bipartite}
	$f(H)=|E(H)|$ if and only if $H$ has no directed path of length $2$ and no directed cycle of length $2$.
\end{proposition}
\begin{proof}
	Suppose first that $H$ has no directed path of length $2$ and no directed cycle of length $2$.
	Then $H$ is an orientation of an undirected bipartite graph where all arcs are oriented from one part to the other part. So, $H$ is a subgraph of $B(L)$ and therefore $\D_H(L)=0$
	implying that $f(H,L)=|E(H)|$ and that $f(H)=|E(H)|$.
	If $H$ has a directed path of length $2$ or a directed cycle of length $2$, then consider $L=T_2$. As $B(T_2)$ has no path of length $2$ and no directed cycle of length $2$,
	we have that $\D_H(T_2)\ge 1$, and so $f(H) \le f(H,T_2) \le \max\{\frac{3}{4}|E(H)|,|E(H)|-1\}$.
\end{proof}
Proposition \ref{prop:bipartite} is in sync with Theorem \ref{t:main-2} in the sense that $f(H)$ in the statement of Theorem \ref{t:main-2} cannot be
replaced by a smaller constant which depends only on $H$ for any given directed graph $H$ with no directed path of length $2$ and no directed cycle of length $2$. Indeed, let $D$ be an orientation
of $K_{n,n}$ where all arcs go from one part to the other and where $n \ge |V(H)|$.
Recalling that the Tur\'an number of (undirected) bipartite graphs is $o(n^2)$, we have that
$\tau_H(D) =n^2(1-o_n(1))$ while $\nu^*_H(D)=n^2/|E(H)|$.

\vspace{0.3cm}
In some cases the infimum in the definition of $f(H)$ is not attained by any $L$.
Although there are infinitely many examples, the simplest is $H=\overrightarrow{C_2}$.
On the one hand, $f(\overrightarrow{C_2},L) > 1$ for any $L$. Indeed, if $L$ has a
directed cycle of length $2$ then $f(\overrightarrow{C_2},L) = 2$. Otherwise,
$\D_{\overrightarrow{C_2}}(L)=1$ and $L$ is a subgraph of some tournament on $r$ vertices
so $f(\overrightarrow{C_2},L) \ge 2(1-r(r-1)/2r^2)=1+1/r$. If $L$ is a tournament
then $f(\overrightarrow{C_2},L) =1+1/r$. Taking $r$ to infinity, we have that $f(\overrightarrow{C_2})=1$.

\vspace{0.3cm}
Let $\gamma(H)$ denote the maximum number of arcs in an acyclic subgraph of $H$.
Equivalently, a {\em minimum feedback arc set} is a set of $|E(H)|-\gamma(H)$ arcs of $H$ whose removal makes $H$ acyclic. It is not difficult to show that for every directed graph $H$,
$\gamma(H) \ge |E(H)|/2$ where equality holds if and only if each pair of vertices of $H$
either induce a directed cycle of length $2$ or an empty graph.
Let $b(H)$ be the maximum number of arcs in a bipartite subgraph of $H$.
Clearly $b(H) > |E(H)|/2$.
\begin{lemma}\label{l:fas}
	$|E(H)|/2 \le f(H)\le \min\{\gamma(H)\,,\, b(H)\}$\;.
\end{lemma}
\begin{proof}
	Let $L=T_r$ where $r \ge |V(H)|$. By the definition of $\gamma(H)$ we have that
	$\D_H(T_r)=|E(H)|-\gamma(H)$.
	We therefore have $f(H,T_r)=\max\{|E(H)|(1-r(r-1)/2r^2),\,\gamma(H)\}$.
	Taking $r$ to infinity we obtain $f(H) \le \gamma(H)$.
	
	Let $L=\overrightarrow{C_2}$. By the definition of $b(H)$ we have
	that $\D_H(\overrightarrow{C_2})=|E(H)|-b(H)$. We therefore have $f(H,\overrightarrow{C_2})=\max\{|E(H)|(1-2/4),\,b(H)\}=b(H)$
	whence $f(H) \le b(H)$.
	
	For the lower bound, consider any nonempty directed graph $L$. Consider first the case where $L$
	has a directed cycle of length $2$. Since every $H$ has a bipartite subgraph containing at least half
	of its arcs, and since any bipartite subgraph of $H$ is a subgraph of $B(L)$ (as $L$ has a directed cycle of length $2$) we have that $\D_H(L) \le |E(H)|/2$ so $f(H,L) \ge |E(H)|/2$.
	If $L$ has no directed cycle of length $2$ then $|E(L)| \le r(r-1)/2$ so
	$f(H,L) \ge |E(H)|(1-r(r-1)/2r^2) \ge |E(H)|/2$.
\end{proof}
\begin{proof}[Proof of Proposition \ref{prop:1}]
Suppose that $G$ is an undirected graph with $n$ vertices and $\Omega(n \ln n)$ edges.
Let $H$ be obtained by randomly and independently orienting each edge of $G$. It is well-known (and a simple exercise to prove) that $\gamma(H)=(1+o_n(1))|E(H)|/2$ asymptotically almost surely.
By Lemma \ref{l:fas} we obtain that asymptotically almost surely, $f(H) = (1+o_n(1))|E(H)|/2$.
\end{proof}

In some cases, as well as some classes of directed graphs, Lemma \ref{l:fas} is far from tight.
Consider the class of directed cycles.
Observe that $\gamma(\overrightarrow{C_k})=k-1$ and $b(\overrightarrow{C_k}) \ge k-1$ so Lemma \ref{l:fas} (while tight for $k=2$) gives a very poor upper bound for $f(\overrightarrow{C_k})$.
The following proposition determines $f(\overrightarrow{C_k})$.
\begin{proposition}\label{p:cycles}
	For all $k \ge 3$ we have $f(\overrightarrow{C_k}) = 2k/3$ unless $k=5$ where $f(\overrightarrow{C_5}) = \frac{25}{8}$.
\end{proposition}
\begin{proof}
	Any directed path in $B(T_r)$ has length at most $r-1$. So, in order to obtain a directed $k$-cycle
	in $B(T_r)$ one must add at least $\lceil k/r \rceil$ arcs.
	Thus, $\D_{\overrightarrow{C_k}}(T_r)=\lceil k/r \rceil$ and  therefore
	$f(\overrightarrow{C_k},T_r) = \max\{k(\frac{1}{2}+\frac{1}{2r}),k-\lceil k/r \rceil\}$\;.
	Using $r=3$ we obtain that $f(\overrightarrow{C_k}) \le 2k/3$ and
	when $k=5$ we can use $r=4$ to obtain $f_{\overrightarrow{C_5}} \le \frac{25}{8}$.
	
	We prove that the upper bound $2k/3$ is tight for all even $k \ge 4$, $k \neq 5$. A similar argument shows tightness for the $25/8$ bound in the case $k=5$. So let $k \ge 4$, $k \neq 5$ and consider some nonempty directed graph $L$. If $L$ has a directed cycle of length $2$ then $\overrightarrow{C_k}$
	is a subgraph of $B(L)$ so we have $\D_{\overrightarrow{C_k}}(L)=0$
	and $f(\overrightarrow{C_k},L)=k$. So, we may assume that $L$ is an orientation.
	If $L$ has a directed path of length $3$ then $\D_{\overrightarrow{C_k}}(L) \le \lceil k/4 \rceil$ implying that $f(\overrightarrow{C_k},L) \ge k - \lceil k/4 \rceil \ge 2k/3$.
	Otherwise, the underlying graph of $L$ does not have a $K_4$, so $|E(L)| \le r^2/3$
	and therefore $f(\overrightarrow{C_k},L) \ge 2k/3$ as well.
\end{proof}

\section{Fractional packing and integral covering}\label{sec:proof}

Throughout this section, let $H$ be a given directed graph with at least two arcs.
We need the following simple lemma, analogous to Lemma 3 of \cite{CGKK-2022}.
\begin{lemma}\label{l:upper}
Let $D$ be an arc-weighted directed multigraph with weight function $w$, let $c:E(D)\rightarrow[0,1]$
be an optimal fractional $H$-cover of $D$, and let $\alpha > 0$. Suppose that there exists an arc $e$
such that $c(e) \ge \alpha > 0$. If $\tau_H(D \setminus e) \le \alpha^{-1}\nu_H^*(D \setminus e)$,
then $\tau_H(D) \le \alpha^{-1}\nu_H^*(D)$.
\end{lemma}
\begin{proof}
	Since $c$ restricted to $E(D) \setminus \{e\}$ is a fractional $H$-cover of
	$D \setminus e$, it follows that
	$$
	\tau_H^*(D \setminus e) \le \tau_H^*(D)-c(e)w(e) \le \tau_H^*(D)  - \alpha w(e)\;.
	$$
	In particular, $\alpha^{-1}\tau_H^*(D\setminus e) + w(e) \le \alpha^{-1}\tau_H^*(D)$.
	
	By the assumption of the lemma, there exists a set $F$ of arcs of weight at most
	$\alpha^{-1}\nu_H^*(D \setminus e)=\alpha^{-1}\tau_H^*(D \setminus e)$ such that $F$ is an $H$-cover of $D \setminus e$. Since the set $F \cup \{e\}$ is an $H$-cover of $D$ and its weight is at most $\alpha^{-1}\tau_H^*(D \setminus e) + w(e) \le \alpha^{-1}\tau_H^*(D) = \alpha^{-1}\nu_H^*(D)$, the lemma follows.
\end{proof}

Let $L$ be a given nonempty directed graph with $r \coloneqq|V(L)|$, $\ell \coloneqq|E(L)|$, and
assume that $V(L)=[r]$. Let
$$
\alpha=\frac{1}{|E(H)|-\D_H(L)}
$$
and observe that $0 < \alpha \le 1$ since $0 \le \D_H(L) < |E(H)|$ as $L$ is nonempty.
\begin{lemma}\label{l:cover}
Let $D$ be an arc-weighted directed multigraph and $c:E(D)\rightarrow[0,1]$
be a fractional $H$-cover of $D$ such that $c(e) < \alpha$ for every arc $e$.
Let $V_1,\ldots,V_r$ be a partition of $V(D)$  (some parts may be empty).
Let $F$ be the set of all arcs $e=(x,y)$ with $c(e) > 0$ and that further satisfy the following:
If $x \in V_i$ and $y \in V_j$ (possibly $i=j$) then $(i,j) \notin E(L)$.
Then $F$ is an $H$-cover of $D$.
\end{lemma}
\begin{proof}
	Let $F^*$ be the set of all arcs $e=(x,y)$ that satisfy the following:
	If $x \in V_i$ and $y \in V_j$ (possibly $i=j$) then $(i,j) \notin E(L)$.
	Observe that $F \subseteq F^*$ and that $e \in F^* \setminus F$ has $c(e)=0$.
	By the definition of $F^*$, the set of arcs $E(D) \setminus F^*$ is a subgraph of $B(L)$.
	Let $X$ be some $H$-copy in $D$. Then $E(X) \setminus F^*$ is a subgraph of $B(L)$, so
	by the definition of $\D_H(L)$, we have that $|E(X) \cap F^*| \ge \D_H(L)$.
	Since $c(e) < \alpha$ for every arc $e$, it cannot be that $\D_H(L)$ arcs of $E(X) \cap F^*$ all have $c(e)=0$ as otherwise the total value of $c$ over all arcs of $X$ is less than $\alpha(|E(H)|-\D_H(L))=1$,
	contradicting the assumption that $c$ is a fractional $H$-cover of $D$.
	It therefore follows that $|E(X) \cap F| > 0$.
\end{proof}

\begin{proof}[Proof of Theorem \ref{t:main-2}]
Let $c$ be an optimal fractional $H$-cover of $D$ and let $m$ be an
optimal fractional $H$-packing. We will show that there exists an $H$-cover with
total value at most $f(H,L)\nu_H^*(D)$.
Using induction on the number of edges of $D$, observe that the theorem trivially holds when $D$ is empty.
By Lemma \ref{l:upper}, we can assume that
$c(e) < f(H,L)^{-1} \le \alpha$ for every arc $e \in E(D)$, as otherwise we can repeatedly apply
Lemma \ref{l:upper} and the induction hypothesis, removing edges of weight at least $f(H,L)^{-1}$ until none are left.

Randomly partition $V(D)$ into $r$ parts $V_1,\ldots,V_r$ where each vertex chooses its part uniformly
at random and independently of other vertices. Using the obtained random partition, we apply
Lemma \ref{l:cover} to obtain an $H$-cover $F$.

Next, we upper-bound the expected weight of $F$, i.e. the sum of the weights of its arcs.
First observe that by the definition of $F$, all arcs $e \in F$ have $c(e) > 0$.
Consider some arc $e=(x,y) \in E(D)$ with $c(e) > 0$. The probability that $e \notin F$
is precisely the probability that $x \in V_i$, $y \in V_j$ and $(i,j) \in E(L)$.
Equivalently, $\Pr[e \in F]=1-\ell/r^2$.
By complementary slackness, we have that if $c(e)>0$, then the
sum of $m(X)$ over all $H$-copies $X$ in $D$ for which $e \in E(X)$ equals $w(e)$.
The expected weight of $F$ is therefore
\begin{align*}
\left(1-\frac{\ell}{r^2}\right)\sum_{\substack{e \in E(D)\\ c(e) > 0}}w(e) & =
\left(1-\frac{\ell}{r^2}\right)\sum_{\substack{e \in E(D)\\ c(e) > 0}}\sum_{\substack{X \in H(D)\\ e \in E(X)}}m(X) \\
& \le |E(H)| \left(1-\frac{\ell}{r^2}\right)\sum_{X \in H(D)}m(X) \\
& \le f(H,L) \nu_H^*(D)\;.
\end{align*}
Thus, there exists a choice of $F$ such that $|F| \le  f(H,L) \nu_H^*(D)$ and in particular,
$\tau_H(D) \le f(H,L) \nu_H^*(D)$. Now, let $\varepsilon > 0$. By the definition of $f(H)$, there exists
a nonempty directed graph $L$ such that $f(H,L) \le f(H)+\varepsilon$, so we have that
$\tau_H(D) \le (f(H)+\varepsilon)\nu_H^*(D)$. As this holds for all $\varepsilon > 0$, we obtain that
$\tau_H(D) \le f(H)\nu_H^*(D)$, as required.
\end{proof}

\begin{proof}[Proof of Corollary \ref{coro:approx-2}]
	To obtain a deterministic polynomial time algorithm for approximating $\tau_H(D)$, we
	compute $\nu_H^*(D)$ using any polynomial time algorithm for linear programming.
	By Theorem \ref{t:main-2}, the approximation ratio is at most $f(H)$.
	
	For the second part of the corollary, first construct (using linear programming) an optimal fractional cover $c$, so its total value is $\tau_H^*(D)=\nu_H^*(D)$.
	Let $L$ be any fixed nonempty directed graph. We compute $\D_H(L)$ in constant time
	since in order to determine $\D_H(L)$ it suffices to consider only induced subgraphs of the blowup $B(L)$ with at most $|V(H)|$ vertices in each part. With $\D_H(L)$ given, we compute
	$f(H,L)$ in constant time.
	By Lemma \ref{l:upper}, we can eliminate from $D$ all arcs with $c(e) \ge f(H,L)^{-1}$ so we can
	now assume that all arcs have $c(e) < f(H,L)^{-1}$.
	By the proof of Theorem \ref{t:main-2}, the
	random set $F$ (which is constructed in linear time as $L$ is fixed), has expected
	weight at most $f(H,L)\nu_H^*(D)$, so we return $F$, which is an $H$-cover, as our algorithm's answer.
	This gives a randomized polynomial time $f(H,L)$-approximation algorithm for $H$-cover.
	To make our algorithm deterministic, we use the derandomization method of conditional expectation.
	Indeed, observe that the precise expected value
	$f(H,L)\nu^*_H(D)$ is known to us. Now, when we construct $F$, we consider the vertices $v \in V(D)$ one by one. In order to decide in which part $V_i$ to place $v$, we simply compute the conditional expectation
	of the expected value of $|F|$ for each of the possible $r$ choices. As one of these choices must yield a value at most $f(H,L)\nu^*_H(D)$ for the conditional expectation, we take that choice.
\end{proof}

\begin{corollary}\label{coro:graph}
	Let $G$ be an edge-weighted undirected multigraph. Then, $\tau_{K_k}(G) \le \lfloor k^2/4\rfloor\nu^*_{K_k}(G)$.
\end{corollary}
\begin{proof}
	Let $H$ be a tournament on $k$ vertices. Clearly, $b(H)=\lfloor k^2/4 \rfloor$ so by Lemma
	\ref{l:fas} we have that $f(H) \le \lfloor k^2/4 \rfloor$. Now, suppose that $G$ is an undirected
	edge-weighted multigraph and let $D$ be an acyclic orientation of $G$. Then any copy of $K_k$ in $G$ is a copy of $T_k$ in $D$ and thus $\rho_{K_k}(G)=\rho_{T_k}(D)$ for any $\rho \in \{\tau,\nu,\tau^*,\nu^*\}$.
	In particular, we obtain from Theorem \ref{t:main-2} that $\tau_{K_k}(G) \le \lfloor k^2/4 \rfloor\nu_{K_k}^*(G)$. Furthermore, Corollary \ref{coro:approx-2} shows that there is a
	polynomial time $\lfloor k^2/4 \rfloor$-approximation algorithm for $K_k$-cover.
\end{proof}

\section{Lower bound construction for directed cycles}\label{sec:lower}
Before presenting the construction which proves Theorem \ref{t:lower}, we need the following result of
H\"aggkvist and Thomassen \cite{HT-1976}. For completeness, we present a simplified proof of it.
We mention that the case $k=3$ of the following lemma was first proved Brown and Harary \cite{BH-1969}.
\begin{lemma}\label{l:HT}
Let $k \ge 2$ and let $D$ be a directed graph with $n$ vertices. If $D$ has no directed $k$-cycle, 
then $D$ has at most $n(n-1)/2+(k-2)n/2$ arcs.
\end{lemma}
\begin{proof}
	Fixing $k \ge 3$ (the case $k=2$ is trivial), the proof proceeds by induction on $n$.
	As the cases $n \le k-1$ clearly hold, we assume that $n \ge k$.
	Since every $n$-vertex undirected graph with more than $n(k-2)/2$ edges has a path on $k$ vertices, we may assume that $D$ has a path $P=v_1,\ldots,v_k$ such that all consecutive pairs
	on this path induce directed cycles of length $2$. Furthermore, if the subgraph induced by $v_1,\ldots,v_k$ does not contain
	a directed $k$-cycle, then the sum of the out-degrees of $v_1$ and $v_k$ inside this subgraph is
	at most $k-1$ and the sum of the in-degrees of $v_1$ and $v_k$ inside this subgraph is
	at most $k-1$. So, without loss of generality, we can assume that in the subgraph induced by $P'=v_2,\ldots,v_k$, the number of arcs incident with $v_k$ is at most $k-1$. The number of arcs incident with either $v_2$ or $v_k$ in $P'$ is therefore at most $(k-1)+2(k-3)=3k-7$.
	If there is some vertex outside of $P'$ that is an in-neighbor of $v_2$
	and an out-neighbor of $v_k$ or vice versa, we have a directed $k$-cycle in $D$.
	Thus, assume that the sum of the in-degree of $v_2$ and the out-degree of $v_k$ with respect to the vertices outside of $P'$ is at most $n-k+1$. Similarly, the sum of the in-degree of
	$v_k$ and the out-degree of $v_2$ with respect to the vertices outside of $P'$ is at most $n-k+1$.
	Thus, the total number of arcs incident with $v_2$ or $v_k$ in all of $D$
	is at most $2(n-k+1)+3k-7=2n+k-5$. By induction, the directed graph obtained from
	$D$ by deleting the vertices $v_2$ and $v_k$ either has a directed $k$-cycle, or has at most
	$(n-2)(n-3)/2+(k-2)(n-2)/2$ arcs. It follows that the number of arcs of $D$ is at most
	$$
	\frac{(n-2)(n-3)}{2}+\frac{(k-2)(n-2)}{2}+2n+k-5=\frac{n(n-1)}{2}+\frac{(k-2)n}{2}\;.
	$$
\end{proof}

We construct a probability space of tournaments having the property that a sampled element of it
satiates the statement of Theorem \ref{t:lower}. We require a classical theorem of Wilson \cite{wilson-1975}
that proves, in particular, that for all sufficiently large $n$ satisfying $n \equiv 1 \pmod {2k}$,
the edges of $K_n$ can be decomposed into pairwise edge-disjoint copies of $C_k$. Given such an $n$ and a decomposition of its edges into a set $\cal{C}$ of
edge-disjoint copies of $C_k$, independently orient each element of $\cal{C}$ 
to obtain a directed $k$-cycle, where each of the two possible directions is chosen at random.
The obtained $n$-vertex tournament $T$ is therefore regular and, by definition,
$\nu_k(T)=n(n-1)/2k$. As trivially $\nu^*_k(T) \le |E(K_n)|/|E(C_k)|=n(n-1)/2k$, we also have
$\nu^*_k(T) = n(n-1)/2k$. We next show that asymptotically almost surely, $\tau_k(T) = n^2/4-o(n^2)$,
thus proving Theorem \ref{t:lower}. Since every directed graph has an acyclic subgraph consisting of at least half of its arcs, it suffices to prove that asymptotically almost surely, $\tau_k(T) \ge n^2/4-o(n^2)$. To this end, we need the following lemma in which the notation $e(A,B)$ denotes the number of arcs of going from
vertex set $A$ to vertex set $B$.
\begin{lemma}\label{l:dense-pairs}
	Asymptotically almost surely, for every pair of disjoint sets $A,B$ of vertices of $T$ of order at least $n^{2/3}$ each, both $e(A,B)$ and $e(B,A)$ are at most $(1+o_n(1))|A||B|/2$.
\end{lemma}
\begin{proof}
	We prove that $e(A,B)$ is tightly concentrated around its expected value, $|A||B|/2$.
	Let $\cal{C}' \subseteq \cal{C}$ be the set of elements of $\cal{C}$ containing at least one edge with endpoints in both $A$ and $B$. Every $C \in \cal{C}'$, being a copy of $C_k$, contains some $1 \le r \le k$ edges with endpoints in both $A$ and $B$. When orienting $C$ to obtain a directed $k$-cycle, some $0 \le s \le r$ of its edges
	become arcs going from $A$ to $B$ and the remaining $r-s$ edges become arcs going from $B$ to $A$,
	or vice versa. Thus, we may associate $C$ with the random variable $X_C$ such that
	$X_C=s-r/2$ with probability $\frac{1}{2}$ and $X_C=r/2-s$ with probability $\frac{1}{2}$
	noticing that
	$$
	e(A,B) = \frac{|A||B|}{2}+\sum_{C \in \cal{C}'} X_C\;.
	$$
	We observe that the $|\cal{C}'|\le |A||B|$ random variables $X_C$ are independent, each having expectation $0$ and $|X_C|=|r/2-s|< k$. So, by the Chernoff inequality A.1.16 in \cite{AS-book},
	$$
	\Pr\left[\sum_{C \in \cal{C}'} X_C > k(|A||B|)^{0.9}\right] \le e^{-(|A||B|)^{1.8}/2|\cal{C}'|}
	\le e^{-(|A||B|)^{0.8}/2} < \frac{1}{5^n}
	$$
	where in the last inequality we have used that $|A||B| \ge n^{4/3}$.
	Thus, with probability at least $1-1/5^n$, $e(A,B) \le (1+o_n(1))|A||B|/2$.
	As there are less than $4^n$ choices for pairs $A,B$ to consider, the result follows from the union bound.
\end{proof}
The rest of our argument is similar to the proof in \cite{CGKK-2022} for directed triangles.
As the proof uses the regularity lemma for directed graphs, it requires a few definitions.
We say that a pair of disjoint nonempty vertex sets $A$, $B$ of a directed graph are {\em $\varepsilon$-regular}
if for all $X \subseteq A$ and $Y \subseteq B$ with $|X|\ge \varepsilon|A|$ and $|Y|\ge \varepsilon |B|$,
$$
\left|\frac{e(X, Y)}{|X||Y|}-\frac{e\left(A, B\right)}{\left|A\right|\left|B\right|}\right| \leq \varepsilon \text { and }\left|\frac{e(Y, X)}{|X||Y|}-\frac{e\left(B, A\right)}{\left|A\right| \left|B\right|}\right| \leq \varepsilon\;.
$$
An $\varepsilon$-regular partition of a directed graph $D$ is a partition of its vertices into sets $V_1, \ldots , V_\ell$ such that $\ell \ge \varepsilon^{-1}$, $||V_i|-|V_j||\le 1$ for all $i,j\in[\ell]$,
and all but $\varepsilon\ell^2$ pairs $V_i,V_j$ are $\varepsilon$-regular.
The directed version of Szemer\'edi's regularity lemma, first used implicitly in \cite{CPR-1999} and proved in \cite{AS-2004}, states that for every $\varepsilon > 0$ there exists $K(\varepsilon)$ such that every directed graph $D$ with at least $\varepsilon^{-1}$ vertices has an $\varepsilon$-regular partition with at most $K(\varepsilon)$ parts. A useful notion is the {\em reduced arc-weighted directed graph} $R$ corresponding to a given $\varepsilon$-regular partition. It has vertex set $[\ell]$ and if the parts $V_i$, $V_j$ form an
$\varepsilon$-regular pair, then $R$ contains an arc $(i,j)$ with weight $e(V_i,V_j)/(|V_i||V_j|)$
and an arc $(j,i)$ with weight $e(V_j,V_i)/(|V_i||V_j|)$. 

\begin{proof}[Proof of Theorem \ref{t:lower}]
	We prove that asymptotically almost surely, $\tau_k(T) \ge n^2/4-o(n^2)$. 
	Fix $\varepsilon > 0$. By Lemma \ref{l:dense-pairs}, we may assume that $T$ has the property
	that for every pair of disjoint sets $A,B$ of vertices of $T$ of order at least $n^{2/3}$ each, it holds that $e(A,B) \le (1+o_n(1))|A||B|/2$ and $e(B,A) \le (1+o_n(1))|A||B|/2$.
	Let $F$ be a set of arcs such that $T \setminus F$ has no directed $k$-cycle.
	Consider an $\varepsilon$-regular partition of the directed graph $T \setminus F$  with $\ell \le K(\varepsilon)$ parts and the corresponding reduced arc-weighted directed graph $R$.
	Let $w_R$ be the sum of the weights of the arcs of $R$. Observe that
	$$
	|E(T \setminus F)| \le \left(\frac{w_R}{\ell^2}+4\varepsilon \right)n^2
	$$
	where the error term $4\varepsilon n^2$ generously accounts for the arcs inside parts and the arcs between non-$\varepsilon$-regular pairs (we are using the fact that each part is of size either
	$\lfloor n/\ell \rfloor$ or
	$\lceil n/\ell \rceil$ and that $\ell \ge \varepsilon^{-1}$).
	Let $R'$ be the directed graph obtained from $R$ by removing all arcs with weight at most $k\varepsilon$
	so now the sum of the weights of the arcs of $R'$ is at least $w_R-k\varepsilon \ell^2$.
	Now, if $R'$ contained a directed $k$-cycle, then so would $T \setminus F$. Indeed, suppose, without loss of generality, that the $k$-cycle in $R'$ is $(1,\ldots,k)$. Then we can use the $\varepsilon$-regularity of the pairs $V_i,V_{i+1}$ for $i=1,\ldots,k$ (indices modulo $k$) and the fact that
	$e(V_i,V_{i+1}) \ge k\varepsilon |V_i||V_{i+1}|$ to embed (many) directed $k$-cycles in
	$T \setminus F$, each of the form $(v_1,\ldots,v_k)$ where $v_i \in V_i$.
	Hence, $R'$ has no directed $k$-cycle and therefore has at most $\ell^2/2+\ell k$ arcs by Lemma \ref{l:HT}. Now, by the property of $T$ stated in the beginning of the proof, each arc of $R$
	has weight at most $1/2+o_n(1)$. It follows that
	$$
	w_R \le k\varepsilon \ell^2 + (\ell^2/2+\ell k)(1/2+o_n(1)) \le \left(\frac{1}{4}+2k\varepsilon \right)\ell^2
	$$
	implying that $|E(T \setminus F)| \le (1/4 +4k\varepsilon)n^2$, implying that
	$|F| \ge n^2(1/4-4k\varepsilon-o_n(1))$. As this holds for every choice of $F$ which covers all directed $k$-cycles, we obtain that $\tau_k(T) \ge (1/4-4k\varepsilon-o_n(1))n^2$, for every $\varepsilon > 0$.
	It follows that $\tau_k(T) \ge n^2/4-o(n^2)$.
\end{proof}

It should be noted that in order to prove that the constant in Conjecture \ref{conj:1} cannot be made smaller than $k/2$, it suffices to prove, say, that there are $n$-vertex tournaments $T$ (not necessarily regular
tournaments) for which $\tau_k(T) \ge n^2/4-o(n^2)$ as trivially 
$\nu^*_k(T) \le n(n-1)/2k$ for every tournament. In fact, almost all tournaments
are good examples, as a random tournament (where each arc is independently and randomly oriented) satisfies $\tau_k(T) \ge n^2/4-o(n^2)$ asymptotically almost surely.
The proof is identical to the proof of Theorem \ref{t:lower} except for Lemma \ref{l:dense-pairs}
which can be replaced with a standard concentration inequality for the binomial distribution.
We also note that it is not difficult to prove that random tournaments satisfy $\nu_k(T) = (1-o_n(1))n^2/2k$ asymptotically almost surely (so they cannot be used as counter-examples to Conjecture \ref{conj:1}).

Both \cite{CGKK-2022,MPT-2020} constructed sparse examples exhibiting the sharp tightness of Conjecture \ref{conj:1} in the case $k=3$ of directed triangles (recall again that the case $k=3$ of Conjecture \ref{conj:1} is stated in \cite{MPT-2020}).
For example, the unique regular tournament
$R_5$ on five vertices has $\nu_3(R_5)=2$ and $\tau_3(R_5)=3$. One can then take many vertex-disjoint
copies of $R_5$ to obtain infinitely many sparse constructions attaining the ratio $1.5$.
Alternatively one can take a transitive tournament on any amount of vertices and replace any number
of pairwise vertex-disjoint subtournaments on five vertices of it with copies of $R_5$ to obtain additional examples attaining the $1.5$ ratio. We note that a similar argument holds for the case $k=4$. Indeed, $\nu_4(R_5)=1$ (since $K_5$ does not have two edge-disjoint copies of $C_4$). While any single arc of $R_5$ does not cover all directed $4$-cycles, it is easy to check that one can remove two arcs and cover all directed $4$-cycles of $R_5$.
Hence, $\tau_4(R_5)=2$. It follows that there are infinitely many constructions that attain the ratio $2$ for the case $k=4$. Whether there exist constructions attaining the exact ratio $k/2$ for $k \ge 5$ remains open.

\section*{Acknowledgment}
I thank both referees for insightful comments.

\end{document}